\documentclass[a4paper,11pt,twoside]{amsart}

\usepackage{amssymb,amsthm,amsmath,amstext}
\usepackage{hyperref,cite,mdwlist}
\usepackage{mathrsfs}  % allows nice script font for sheaves
\usepackage{bm}        % allows bold italic letters in math mode
\usepackage{mathabx} %gets us the big cartesian product
\usepackage{multirow}
\usepackage[top=1in,bottom=1in,left=1in,right=1in]{geometry}
\usepackage{color}

\usepackage[all]{xy}
\usepackage{rotating}
\usepackage[all]{xy}
\usepackage{rotating}

\newtheorem{theorem}{Theorem}[section]
\newtheorem{proposition}[theorem]{Proposition}
\newtheorem{lemma}[theorem]{Lemma}
\newtheorem{corollary}[theorem]{Corollary}
\newtheorem{conjecture}[theorem]{Conjecture}

\newtheorem{thmintro}{Theorem}

\DeclareMathOperator{\Rep}{\mathrm{rdim}}
\DeclareMathOperator{\Repcat}{\mathrm{rdim}}
\DeclareMathOperator{\coRepcat}{\mathrm{rcodim}}

\theoremstyle{definition}
\newtheorem{definition}[theorem]{Definition}
\newtheorem{remark}[theorem]{Remark}
\newtheorem{example}[theorem]{Example}
\newtheorem{question}[theorem]{Question}

\DeclareMathOperator{\Chow}{Chow}
\DeclareMathOperator{\SmProj}{SmProj}

\newcommand{\op}{\mathsf{op}}

\newcommand{\sod}[1]{\langle #1 \rangle}
\newcommand{\too}{\longrightarrow}

\newcommand{\Db}{{\rm D}^{\rm b}}

\newcommand{\Aut}{{\rm Aut}}

\newcommand{\Br}{{\rm Br}}

\newcommand{\Hom}{{\rm Hom}}
\newcommand{\Spec}{{\rm Spec}}

\renewcommand{\dim}{{\rm dim}\,}

\newcommand{\dgcat}{\mathsf{dgcat}}

\newcommand{\id}{{\rm id}}

\newcommand{\cal}{\mathcal}

\newcommand{\ka}{{\cal A}}
 
\newcommand{\kb}{{\cal B}}
  
\newcommand{\kc}{{\cal C}}

\newcommand{\ko}{{\cal O}}

\newcommand{\NN}{\mathbb{N}}
\newcommand{\ZZ}{\mathbb{Z}}
\newcommand{\QQ}{\mathbb{Q}}

\newcommand{\CC}{\mathbb{C}}

\newcommand{\LL}{\mathbb{L}}
\newcommand{\PP}{\mathbb{P}}

\newcommand{\Bir}{\mathrm{Bir}}

\newcommand{\linedef}[1]{\textit{#1}}

\usepackage{hyperref}

\begin{document}
\title[Categorical dimensions and filtrations of Cremona groups]{Categorical dimension of birational transformations and filtrations of Cremona groups}

\author{Marcello Bernardara}
\address{Institut de Math\'ematiques de Toulouse \\ %
Universit\'e Paul Sabatier \\ %
118 route de Narbonne \\ %
31062 Toulouse Cedex 9\\ %
France}
\email{marcello.bernardara@math.univ-toulouse.fr}

\begin{abstract}
Using a filtration on the Grothendieck ring of triangulated categories, we define the motivic categorical dimension
of a birational map between smooth projective varieties. We show that birational transformations of bounded motivic categorical
dimension form subgroups, which provide a nontrivial filtration of the Cremona group. We discuss
some geometrical aspect and some explicit example.
We can moreover define, in some cases, the genus of a birational transformation, and compare it to the one defined by Frumkin
in the case of threefolds.
\end{abstract}

\maketitle

\section{Introduction}

In the last decades, derived categories of coherent sheaves and their
semiorthogonal decompositions attracted a growing amount of interest, thanks, among other features,
to the conjectural interplay with birationality questions.
Since the seminal work of Bondal and Orlov \cite{bondal_orlov:ICM2002},
examples and ideas based on the motivic behavior of semiorthogonal decompositions have lead
to formulate some natural questions about the possibility to define a birational invariant, or, at least, to
obtain necessary conditions for rationality of a given variety from semiorthogonal
decompositions of its derived category. We refer to \cite{kuz:rationality-report} and
\cite{auel-berna-frg} for recent reports on motivations, open questions, conjectures, technical problems,
and comparison to other theories.

In an effort to understand the above questioning, and inspired by the classical
notion of representability of Chow groups, the definition of \linedef{categorical representability}
of a smooth projective variety $X$ was given in \cite{bolognesi_bernardara:representability}.
Roughly speaking, such an $X$ is categorically representable in dimension $m$ if the derived category $\Db(X)$
admits a semiorthogonal decomposition whose components can be realized in semiorthogonal decompositions
of varieties of dimension at most $m$. Based on blow-up formulas and Hironaka resolution of
singularities, the upshot of this definition is to understand whether
being representable in codimension $2$ is a necessary condition for rationality.

\smallskip

A fundamental invariant one can consider to study the birational geometry of a given variety $X$ is the group
of birational transformations $\mathrm{Bir}(X)$. The definitions and the results presented here
came out of an attempt to understand the interplay between semiorthogonal decompositions
and the group $\mathrm{Bir}(X)$.
The guiding idea is the possibility to define, basing upon weak factorization and categorical representability,
the notion of \linedef{motivic categorical dimension} of a given birational map 
$\phi: X \dashrightarrow Y$, see Definition \ref{def:cat-dim-of-birmap}. This is, to the best of the author's knowledge, a first, though
very little, attempt to obtain informations on $\mathrm{Bir}(X)$ using semiorthogonal decompositions, while
their interplay with other commutative birational invariants is
already been treated (see \cite{auel-berna-frg} for a recent report) and sometimes even well understood
(see \cite{at12} for the Hodge theory of cubic fourfolds, or
\cite{berna-tabuada-jacobians} for intermediate Jacobians).

Let us resume the circle of ideas leading to Definition \ref{def:cat-dim-of-birmap}.
Given a weak factorization of $\phi$, one has to consider the centers of all the blow-ups involved,
and calculate the maximal value of categorical representability of such centers. One would like to define
the categorical dimension of $\phi$ to be the minimum of such values on all the possible weak factorizations
of $\phi$. However, one of the main technical (potential) obstructions to have this value well-behaved is the lack of
Jordan-H\"older property for semiorthogonal decompositions, see \cite{kuznet:JH} or \cite{boeh-graf-sos-JH}.
In particular, if we used such a definition, we would not be able to compare the dimensions of $\phi$ and $\phi^{-1}$,
a central result in order to have a group filtration on $\mathrm{Bir}(X)$.

In order to obtain a well-behaved definition, we need to work in the Grothendieck rings $PT(k)$
and $K_0(\mathrm{Var}(k))$ of dg categories
and, respectively, varieties over $k$. Roughly speaking, in these rings, semiorthogonal decompositions and, respectively,
blow-ups turn into sums. Moreover, categorical representability induces a ring filtration $PT_*(k)$ on $PT(k)$.
Bondal, Larsen and Lunts \cite{bondal-larsen-lunts} defined a \linedef{motivic measure} (i.e. a ring
homomorphism) $\mu: K_0(\mathrm{Var}(k)) \to PT(k)$ by sending $[X]$ to the class of $\Db(X)$ for any
smooth projective variety $X$. We define then the \linedef{motivic categorical dimension} of $X$ to be the smallest
$d$ such that $\mu([X])$ lies in $PT_d(k)$. This value is bounded above by the categorical representability of
$X$, but it is not known if they always coincide.
Given a birational map $\phi: X \dashrightarrow Y$, replacing categorical representability by motivic categorical dimension,
we define the \linedef{motivic categorical dimension} of $\phi$ as we sketched above: running through all possible weak
factorizations, we take the minimum value of the maximal motivic categorical dimension of the centers involved
in the factorization.
In the case where $X=Y$ this gives a filtration
of $\mathrm{Bir}(X)$ by setting $\mathrm{Bir}_d(X)$ to be the subset of birational maps of
motivic categorical dimension at most $d$. By setting the motivic categorical dimension of the trivial
category to be $-1$, we have that $\Bir_d(X)$ is defined and not empty for $d \geq -1$ any integer.
Since our definitions assume the existence of a weak factorization for birational maps, we work over a field $k$ of
characteristic zero, where every birational map admits such a factorization \cite{abrmovichandco-weak}.

\begin{thmintro}\label{thm:intromain}
Let $X$ be smooth and projective variety, then $\mathrm{Bir}_d(X)$ is a subgroup, which coincides with the whole
group if $d \geq \dim(X)-2$ and with $\mathrm{Aut}(X)$ if $d=-1$.
\end{thmintro}

A slightly more general version of Theorem \ref{thm:intromain} will be stated and proved as Theorem \ref{thm:main}.

\medskip

Let us illustrate some geometric applications of Theorem \ref{thm:intromain}
in the most interesting cases, that is $X=\PP^n$. We will make use of arguments based on Hodge theory
via Hochschild homology, and we work out then explicit examples for the case of varieties defined over an algebraically
closed subfield $k \subset \CC$. First of all, using Hochschild homology, we can prove that we actually defined a proper
filtration.

\begin{thmintro}\label{thm:introfiltraisgood}
Let $n \geq 3$, and assume $k \subset \CC$ is algebraically closed. Then $\Bir_{d}(\PP^n) \neq \Bir_{d+1}(\PP^n)$ for all $-1 \leq d < n-2$.
\end{thmintro}

Theorem \ref{thm:introfiltraisgood} will be stated and proved as Theorem \ref{prop:arealfiltration}.

Given a birational map, being able to provide a weak factorization should give interesting informations
on the birational map itself. As potential applications of Theorem \ref{thm:introfiltraisgood}, one could try to show that
a given class of birational maps does not generate the whole Cremona group in dimension $n$, by showing that the motivic 
categorical dimension of such maps must be bounded above by $n-3$ or less.
For example, using more refined motivic arguments, we can see that if a birational map
$\phi: X \dashrightarrow X$ admits a weak factorization with rational centers, then $\phi$ belongs to
$\mathrm{Bir}_{n-4}(X)$, see Corollary \ref{cor:rational-centers-dimn-4}.
Finally, if such a $\phi$ admits a weak factorization whose centers
are all (abstractly isomorphic to) toric varieties, one can use a result of Kawamata \cite{kawamata-toric}
to show that $\phi$ belongs to $\Bir_0(X)$. 
It follows that the subgroup of $\Bir(\PP^n)$ generated by the
standard Cremona transformation and transformations is a subgroup of $\Bir_0(\PP^n)$.
Notice that $\Bir_0(\PP^n) \subset \Bir_{n-4}(\PP^n)$ is strict as soon as $n \geq 5$, and that there exist rational
varieties of dimension at least three with positive motivic categorical dimension.

\smallskip

In \cite{frumkin}, Frumkin defined a filtration of the group $\Bir(X)$ for a uniruled complex threefold
$X$. This is done by defining the \linedef{genus} of a birational map $\phi$ to be the maximum of the genera among
the centers of the blow-ups in regular resolutions of $\phi^{-1}$. As shown by Lamy \cite{lamy-genus},
this is exactly the same as the maximum of the genera of irreducible divisors contracted by $\phi$. 
Using the theory of noncommutative motives and the fact, proved in \cite{berna-tabuada-jacobians},
that one can reconstruct intermediate Jacobians and their polarizations via semiorthogonal decompositions,
we can give an alternative definition of the genus of a birational maps $\phi: X \dashrightarrow X$ in the case where
both $X$ and all the centers of a weak factorization of $\phi$ have, roughly speaking, well-behaved
principally polarized intermediate Jacobians. This can be compared to Frumkin's definition,
see Proposition \ref{prop:frumkin}, and has interesting potential interplays with it.

\smallskip

We conclude by recalling that Dmitrov, Haiden, Katzarkov and Konstevich have defined in
\cite{DHKK} the notion of entropy for an endofunctor of a triangulated category. On the
other hand, the entropy of a birational map is a well-known, very interesting object of study.
As asked by the above authors, it would be very interesting to understand if and how these two notions
can be related to each other \cite[\S 4.3]{DHKK}. Here, noncommutative
methods are used to produce a filtration of the group $\Bir(X)$. It is not clear to the author,
whether knowing the motivic categorical dimension of a birational map could give any information
on its topological entropy.

\medskip

We work exclusively over a field $k$ of characteristic zero to ensure the existence
of weak factorizations \cite{abrmovichandco-weak}.

\subsection*{Acknowledgements}
This project started when I learnt about Frumkin's definition of genus of
a birational map. I thank St\'ephane Lamy to let me discover this notion, and himself and Massimiliano Mella
for helpful discussions and comments. I am very
grateful to J\'er\'emy Blanc for pointing out a gap in an argument in an early version,
and for discussions and examples, and to the anonymous referees for useful comments
and remarks.

\section{Noncommutative tools}
\label{subsec:cat_rep}

In this section we aim to give a short introduction to the noncommutative tools that are
actively involved in the next: the Grothendieck ring of triangulated categories and its filtration, and
Hochschild homology. We assume the reader to be
familiar with notions such as semiorthogonal decompositions and noncommutative motives;
references can be found in \cite{auel-berna-frg} and \cite{tabuada-book} respectively.

\subsection{Categorical representability}
Using semiorthogonal decompositions, one can define a notion of
\linedef{categorical representability} for triangulated categories. In
the case of smooth projective varieties, this is inspired by the
classical notions of representability of cycles, see
\cite{bolognesi_bernardara:representability}.

\begin{definition}
\label{def-rep-for-cat}
A triangulated category $\ka$ is \linedef{representable
in dimension $m$} if it admits a semiorthogonal decomposition
$$
\ka = \langle \ka_1, \ldots, \ka_r \rangle,
$$
and for each $i=1,\ldots,r$ there exists a smooth projective
$k$-variety $Y_i$ with $\dim Y_i \leq m$, such that $\ka_i$ is
equivalent to an admissible subcategory of $\Db(Y_i)$.

We use the following notation
$$\Rep(\ka) := \min \{ m \in \NN \, \vert \, \ka \text{ is representable in dimension } m\},$$
whenever such finite $m$ exists.
\end{definition}

\begin{lemma}\label{lem:dim0dim1}
Let $\ka$ be triangulated category.
\begin{enumerate}
 \item $\Rep\ka=0$ if and only if there exists a semiorthogonal
decomposition $$
\ka = \langle \ka_1, \ldots, \ka_r \rangle,
$$
such that for each $i$, there is a $k$-linear equivalence $\ka_i
\simeq \Db(K_i/k)$ for a separable field extension $K_i/k$.

\item $\Rep\ka \leq 1$ if and only if $\ka$ admits
a semiorthogonal decomposition whose components belong to the following list:

\begin{itemize}
 \item categories representable in dimension 0, or
 \item categories of the form $\Db(k,\alpha)$, for $\alpha$ in $\Br(k)$ the Brauer class of a conic, or
 \item categories equivalent to $\Db(C)$ for some smooth $k$-curve $C$.
\end{itemize}

\end{enumerate}

\end{lemma}

\begin{proof}
The Lemma combines \cite[Prop. 6.1.6, Prop. 6.1.10]{auel-berna-frg}.
\end{proof}

\begin{example}[Bondal-Kuznetsov's counterexample]\label{ex:bond-kuz}
Consider a line $l$ in $\PP^3$. On the blow-up of $\PP^3$ along $l$, consider a smooth rational curve $C$ intersecting
the exceptional divisor in two distinct points. Set $X$ to be the iterated blow-up of $\PP^3$ along $l$ and $C$.
As shown by Kuznetsov \cite{kuznet:JH}, the category $\Db(X)$ contains an admissible subcategory $\kb$ which
cannot be generated by exceptional objects, and in particular we have $\Rep(\kb)>0$ (by Lemma \ref{lem:dim0dim1}).
On the other hand,
the description of $X$ as an iterated blow-up of $\PP^3$ along rational curves provides a full exceptional
collection for $X$, so that $\Rep(X)=0$. It follows, that $\Rep$ is not a monotone function with respect to admissible
embeddings.
\end{example}

\begin{definition}
\label{def-cat-rep}
Let $X$ be a smooth projective $k$-variety. We say that $X$
is \linedef{categorically representable} in dimension $m$ (or
equivalently in codimension $\dim(X)-m$) if $\Db(X)$ is representable in
dimension $m$.

We will use the following notations:
$$
\Repcat(X) := \Rep(\Db(X)), \,\,\,\,\,\, \coRepcat(X):= \dim(X)-\Rep(\Db(X)),
$$
and notice that they are both nonnegative integer numbers.
\end{definition}

\begin{lemma}\label{lem:trivial-bound}
Let $X$ be a smooth projective $k$-variety of dimension $n$. Then
$\Repcat(X) \leq n$ and $\coRepcat(X)\geq 0$.
\end{lemma}
\begin{proof}
This is immediate by taking the trivial semiorthogonal decomposition
$\Db(X) = \sod{\Db(X)}$.
\end{proof}

\begin{remark}
The category $\kb$ from Example \ref{ex:bond-kuz} is not
of the form $\Db(X)$ for any smooth and projective variety $X$. Hence, the question
whether $\Repcat$ is monotone on the set of smooth and projective varieties 
with respect to fully faithful functors $\Db(X) \to \Db(Y)$ is open.
\end{remark}

\subsection{The Grothendieck ring of pretriangulated dg categories and its filtration}\label{sect:grothring-noncomm}

We sketch Bondal-Larsen-Lunts construction of the Grothendieck ring of $k$-linear smooth and proper pretriangulated dg categories \cite{bondal-larsen-lunts}.
Recall that a $k$-linear dg category is a category enriched over the category of complexes of $k$-vector spaces and that functors between such
categories are hence enriched functors.
Recall that if $\ka$ is a pretriangulated dg category, then its homotopy category $H^0(\ka)$ is the category with the same objects and whose morphisms
are obtained by taking the $0$-th cohomology of the dg complex. Hence $H^0(\ka)$ is triangulated and $k$-linear.
A semiorthogonal decomposition $\ka=\sod{\kb,\kc}$ of such an $\ka$ is given by a pair of pretriangulated dg subcategories such that
$H^0(\ka)=\sod{H^0(\kb),H^0(\kc)}$ is a semiorthogonal decomposition of the underlying homotopy triangulated category. 
Let $\ZZ[\dgcat]$ be the free $\ZZ$-module generated by equivalence classes of smooth and proper pretriangulated dg categories, and introduce the equivalence relation
generated by:
\begin{equation}\label{eq:def-of-groth-ring}
\begin{array}{lr}
\ka \sim \kb + \kc & \text{if }\,\, \ka = \langle \kb, \kc \rangle.
\end{array}
\end{equation}
We denote the quotient group by $I:\ZZ[\dgcat] \to PT(k)$ (see \cite[\S 5.1]{bondal-larsen-lunts}).
For a smooth projective variety $X$, notice that there is a unique dg enhancement of $\Db(X)$ (see \cite{LO}). We will
often use the notation $I(X)$ for the class $I(\Db(X))$ in $PT(k)$.

\begin{definition}
We set 
$PT_{eff}(k):=I(\ZZ_{\geq 0}[\dgcat])$, and we
say that an element $a$ of $PT(k)$ is \linedef{effective} if it belongs to $PT_{eff}(k)$.
\end{definition}

\begin{remark}
Note that $a$ in $PT(k)$ is effective if and only if there exists a smooth and proper dg category $\ka$ such that $a=I(\ka)$. Indeed,
by definition we can write $a = \sum_{i=1}^r I(\ka_i)$, so that we can set $\ka=\bigoplus_{i=1}^r \ka_i$.
\end{remark}

In the additive commutative group $PT(k)$,
define the following associative product:
\begin{equation}\label{eq:def-of-prod-pt}
I(\ka) \bullet I(\kb)  = I(\ka \otimes \kb).
\end{equation}

\begin{proposition}[\cite{bondal-larsen-lunts}, Cor. 5.7]
The group $PT(k)$ endowed with the product $\bullet$ is
a commutative associative ring with unit $1 = I(\Spec(k))$.
\end{proposition}

The notion of categorical representability induces a ring filtration on
$PT(k)$, as follows.

\begin{definition}
Let $d$ be a nonnegative integer. We set 
$$PT_d(k) := \sod{ a \in PT_{eff}(k) \vert \text{ there exists } \ka \text{ with } \Rep(\ka)\leq d \text{ and } a \text{ is a summand of } I(\ka)}$$
to be the 
additive subgroup generated by effective summands of elements of the
form $I(\ka)$ with $\Rep(\ka)\leq d$.
\end{definition}

We notice that this definition is
slightly different from the one introduced in \cite[\S 8.1]{auel-berna-frg}.

\begin{proposition}\label{prop:the-filtration}
The subsets $PT_i(k)$ give a filtration on the ring $PT(k)$.
More precisely, suppose that $a$ is in $PT_i(k)$
and $b$ is in $PT_j(k)$. Then
$$\begin{array}{l}
a+b \,\, \text{ is in } PT_{\mathrm{max}(i,j)}(k),\\
a \bullet b \,\, \text{ is in } PT_{i+j}(k).
  \end{array}$$
\end{proposition}
\begin{proof}
First of all, by definition, $PT_i(k) \subset PT_{i+1}(k)$ for any integer $i \geq 0$.

To prove that $a+b$ is in $PT_{\mathrm{max}(i,j)}(k)$, it is enough to assume, without loss of generality,
that $i \leq j$, and recall that $PT_j(k)$ is an additive group by definition. 

To prove that $a \bullet b \,\, \text{ is in } PT_{i+j}(k)$, it is enough to work on generators and consider the case where $a=I(\ka)$, and $b=I(\kb)$
for $\ka$ (resp. $\kb$) admissible subcategory of $\Db(X)$ (resp. $\Db(Y)$) and $X$ (resp. $Y$) smooth and projective of dimension at most $i$ (resp. $j$).
Then $\ka \otimes \kb$ is admissible in $\Db(X \times Y)$ (see, e.g., \cite{kuznetbasechange}), so that $\Repcat(\ka\otimes\kb) \leq i+j$ by an
easy dimension calculation.
\end{proof}

\begin{definition}
Let $\ka$ be a pretriangulated dg category. The \linedef{motivic categorical dimension} of $\ka$
is the smallest integer $d$ such that $I(\ka)$ belongs to $PT_d(k)$. We denote this value (which is either
a nonnegative number or infinity), by $\mathrm{mcd}(\ka)$. More explicitly, we have

$$
\begin{array}{rl}
\mathrm{mcd}(\ka) = \mathrm{min} &\{ d \,\, | \,\, I(\ka) \in PT_d(k) \}.
\end{array}$$

If $X$ is a smooth and projective scheme,
then we also set $\mathrm{mcd}(X):=\mathrm{mcd}(\Db(X))$. By convention, we set $\mathrm{mcd}(0)=-1$.

\end{definition}

\begin{lemma}\label{lemma:compare-mcd-catrep}
For any pretriangulated dg category $\ka$, we have:
$$\mathrm{mcd}(\ka) \leq \Rep(\ka).$$

For any smooth and projective variety $X$, we have:
$$\mathrm{mcd}(X) \leq \Rep(X) \leq \dim(X).$$
\end{lemma}

\begin{proof}
The first inequality follows by definition of the subset $PT_d(k)$: if $\Rep(\ka) = d$, then clearly
$I(\ka)$ belongs to $PT_d(k)$, hence $\mathrm{mcd}(\ka) \leq d$. Now, if $X$ is smooth and projective,
the first part of the second inequality is nothing but the previous one with $\ka=\Db(X)$. The second part of the same inequality
is proved in Lemma \ref{lem:trivial-bound}.
\end{proof}

\begin{remark}
We notice that strict inequality $\mathrm{mcd}(\ka) < \Rep(\ka)$
can hold
over any field $k$, as the category $\kb$ from Example
\ref{ex:bond-kuz} shows. In particular, if we set $\dgcat_{\leq d}(k)$ to be the
full subcategory of $\dgcat(k)$ whose objects are the dg categories $\ka$ such that
$\Rep\ka \leq d$, then we have 
$$PT_{eff}(k) \cap PT_d(k) \neq I(\ZZ[\dgcat_{\leq d}(k)]),$$
even for $d=0$.
A natural question is then to find conditions for which the equality
$\mathrm{mcd}(\ka) = \Rep (\ka)$ holds. As we will see in Corollary \ref{cor:maximal-motivic-dimension},
there exist smooth and projective manifolds $X$ with $\mathrm{mcd}(X)=\Repcat(X)=\dim(X)$. For example this holds if
$X$ is Calabi-Yau (see Example \ref{ex:CY_mfds} as well).
\end{remark}

Consider now the Grothendieck ring $K_0(\mathrm{Var}(k))$ of $k$-varieties whose unit $1=[\Spec(k)]$ is the class
of the point, and recall that a \linedef{motivic measure} is a ring homomorphism
$\mu: K_0(\mathrm{Var}(k)) \to R$ to some ring $R$. 
Using weak factorization, the ring $K_0(\mathrm{Var}(k))$ can be seen as the $\ZZ$-module generated by isomorphism classes of smooth proper varieties,
where we set $[\emptyset]=0$ and with the relation $[X] - [Z] = [Y] - [E]$ whenever $Y \to X$ is the blow-up
along the smooth center $Z$ with exceptional divisor $E$, see \cite{bittner}.
The class of the affine line in $K_0(\mathrm{Var}(k))$ is denoted by $\LL$.

Bondal, Larsen, and Lunts \cite{bondal-larsen-lunts} show that, in this case, the assignment $[X] \mapsto I(\Db(X))$ for a smooth projective 
variety $X$ defines a motivic measure:
\begin{equation}\label{eq:def-of-mbll}
\mu: K_0(\mathrm{Var}(k)) \to PT(k).  
\end{equation}
We will denote $I(X):=\mu([X])=I(\Db(X))$.
Recall the semiorthogonal decomposition $\Db(\PP^1) = \sod{\ko, \ko(1)}$, where both components
are equivalent to $\Db(\Spec(k))$. Since $[\PP^1]=\LL + 1$ in $K_0(\mathrm{Var}(k))$, we deduce that
$\mu(\LL)=1$. More generally, $\mu([\PP^m])=m+1$ for any positive integer $m$.

\begin{remark}\label{rmk:noproblem}
Notice that, if $X$ is a smooth projective $k$-scheme, then $I(X) \neq 0$ in $PT(k)$.
This can be shown, for example, using that Hochschild homology (see below) is nontrivial.
Hence, if $X$ and $Y$ are smooth and projective $k$-varieties, and $m$, and $n$ nonnegative
integers, then $m I(X) + n I(Y) = 0$ if and only if $m=n=0$. Indeed, the former is the
class of the scheme $(\PP^{m-1} \times X) \amalg (\PP^{n-1} \times Y)$.
\end{remark}

\subsection{Hochschild homology}\label{sect:Hochhom}
We denote by $\mathsf{dgcat}(k)$ the category of small pretriangulated dg categories with morphisms given
by dg functors.
An additive invariant is a functor from the category
$\mathsf{dgcat}(k)$ to some additive category that send
semiorthogonal decompositions into direct sums. Many of such invariants are the noncommutative
interpretation of well-known cohomology theories. We present here one of them, Hochschild homology,
which can be thought of as a noncommutative interpretation of the (vertically graded) Hodge structure
on Betti cohomology. Hochschild homology will turn out to be a very useful tool in our proofs. However, there
are several more additive invariants that one can consider, see \cite[\S 2]{tabuada-book} for a detailed
account.

% Let $A$ be a dg algebra, and consider $A$ as a bimodule over itself (or, better, as an $A \otimes A^\op$-bimodule)
% in a canonical way. Then define
% $$HH_\bullet(A) := A \otimes_{A \otimes A^\op}^{\mathbf L} A.$$
% Hochschild homology is invariant under Morita equivalence of $A$ (see, e.g., \cite{keller-icm}). More generally (see \cite[\S 2]{tabuada-book}), 

Let us just recall that for any integer $n$, we have functors
$$\begin{array}{rl}
HH_n :& \mathsf{dgcat}(k) \too \mathrm{Vect}(k),\\
\end{array}
$$
and notice that $HH_n$ and $HH_\bullet$ are additive invariants, see, e.g., \cite[\S 2.2.8]{tabuada-book}.
We will also consider the graded $k$-vector space
$$HH_\bullet(\ka) := \bigoplus_{i \in \ZZ} HH_i(\ka),$$
so that $HH_\bullet$ also gives rise to an additive invariant.
If $\ka = \Db(X)$, we will use the shorthands
$HH_i(X):=HH_i(\Db(X))$ and $HH_\bullet(X):=HH_\bullet(\Db(X))$.
We will also use the notation $hh_i(\ka)$ to denote $\mathrm{dim}HH_i(\ka)$, and note that
$hh_i$ descends to a linear function:
\begin{equation}\label{eq:defining-hhi}
hh_i : PT(k) \longrightarrow \ZZ.
\end{equation}

We hence obtain the following proposition either via noncommutative motives, as done by Tabuada \cite{tabuada-book},
or in an explicit geometric way, as done by Kuznetsov \cite{kuz:hochschild}.

\begin{proposition}\label{prop:deco-of-hochschild}
Let $X$ be a smooth and projective variety, and $\ka \subset \Db(X)$ an admissible subcategory.
If $\ka= \sod{\ka_1,\ldots,\ka_r}$ is a semiorthogonal decomposition of $\ka$, we have:
$$\begin{array}{lr}
HH_\bullet(\ka) \simeq \bigoplus_{i=1}^r HH_\bullet (\ka_i), & HH_n(\ka) \simeq \bigoplus_{i=1}^r HH_n (\ka_i)
  \end{array}$$
for every integer $n$.
\end{proposition}

Proposition \ref{prop:deco-of-hochschild}, together with the Hochschild-Konstant-Rosenberg isomorphisms \cite{HKR-original}
is a very useful tool to bound the motivic categorical dimension
of a smooth projective variety.

\begin{proposition}\label{prop:from-rep-to-hoch}
Let $\ka$ be a pretriangulated dg category.
If $\Rep(\ka) = m$, then $HH_i(\ka)=0$ for $|i| > m$.
\end{proposition}
\begin{proof}
Let us first recall that, if $X$ is a smooth projective variety of dimension $m$, the
the Hochschild-Konstant-Rosenberg isomorphism gives $HH_i(X)=0$ for $|i|>m$.
Hence, for any $\kb$ admissible in $\Db(X)$, we apply Proposition \ref{prop:deco-of-hochschild}
to obtain that $HH_i(\kb)=0$ for $|i|>m$. 

Now if $\ka$ satisfies the assumptions, there is a semiorthogonal decomposition $\ka=\sod{\ka_1,
\ldots,\ka_r}$ and with $\ka_j$ admissible in $\Db(X_j)$ and $X_j$ smooth and projective of
dimension $\dim(X_j) \leq m$, for any $j=1,\ldots,r$. It follows that $HH_i(\ka_j)=0$
for $|i|>m$ and any $j=1,\ldots,r$. Finally, we conclude by using Proposition \ref{prop:deco-of-hochschild}
which gives $HH_i(\ka)=\oplus_{j=1}^r HH_i(\ka_j) = 0$ for $|i|>m$.
\end{proof}

If $k \subset \CC$ is an algebraically closed subfield, and $X$ a smooth
projective variety over $k$,
Weibel \cite{weibel-hoch-hodge} has shown that:
\begin{equation}\label{eq:hochtohodge}
HH_n(X) \simeq \bigoplus_{p-q=n} H^{p,q}(X),
\end{equation}
where $H^{p,q}(X)=H^q(X,\Omega_X^p)$. This gives another criterion to bound the motivic categorical
dimension of a smooth projective variety.

\begin{corollary}\label{cor:maximal-motivic-dimension}
If $X$ is a smooth projective variety over an algebraically closed field $k \subset \CC$, set
$$m := \mathrm{max}\{ i \, \vert \text{ there exist } p,q \text{ such that } p-q=i \text{ and } h^{p,q}(X) \neq 0 \}.$$
Then
$$\Repcat(X) \geq \mathrm{mcd}(X) \geq m.$$
In particular, if $X$ has dimension $n$ and $h^{n,0}(X) \neq 0$, then $\Repcat(X)=\mathrm{mcd}(X)=n$.
\end{corollary}
\begin{proof}
The first inequality has already been proved in Lemma \ref{lem:trivial-bound}.
To obtain the second inequality, it is enough to apply Proposition \ref{prop:from-rep-to-hoch} to the category $\Db(X)$ and use the formula
\eqref{eq:hochtohodge} to see that we cannot have $\mathrm{mcd}(X) < m$. The last chain of equalities easily follows from $\Repcat(X) \leq n$.
\end{proof}

\begin{example}\label{ex:CY_mfds}
Let $X \subset \PP^{n+1}$ be a smooth projective hypersurface of degree $n+2$ over an algebraically closed subfield $k \subset \CC$.
Then $\mathrm{mcd}(X)=n$. Indeed, we have that $H^{n,0}(X)$ is one-dimensional.
\end{example}

The above criteria and example will provide a very useful tool in the proof of Theorem \ref{thm:introfiltraisgood}.
However, the converse to the statement in Proposition \ref{prop:from-rep-to-hoch} is not true as soon as $\Rep(\ka) \geq 2$.
The following result shows that categorical representability captures indeed much finer invariants than Hochschild homology.

\begin{proposition}\label{prop:cat2surfaces}
Let $S$ be a smooth complex projective surface with $h^{1,0}(S)=h^{2,0}(S)=0$, and such that $K_0(S)$ has nontrivial torsion.
Then $HH_i(S)=0$ for $i \neq 0$, and $\Repcat(S)=2$.
\end{proposition}
\begin{proof}
The proof can be retraced along \cite[\S 6]{auel-berna-frg}, but let us quickly recall it.
First, $HH_i(S)=0$ since $h^{p,q}(S)=0$ as soon as $p \neq q$, by assumption.

On the other hand,
$\Repcat(S) \leq 2$ by Lemma \ref{lem:trivial-bound}. 
Now, for any smooth projective complex variety $X$, the vanishing $\Repcat(X)=0$ is equivalent
to the existence of a full exceptional collection (see Lemma \ref{lem:dim0dim1}), and this implies that $K_0(X)$ is
free of finite rank. Then $\Repcat(S) \geq 1$.
If $\Repcat(S)=1$, then there is a curve $C$ of positive genus and a fully faithful functor
$\Db(C) \to \Db(S)$, see Lemma \ref{lem:dim0dim1}.
But this would give $0 \neq H^{1,0}(C) \subset HH_1(C) \subset HH_1(S)$, which contradicts
the assumption $h^{1,0}(S)=0$, so that $\Repcat(S)=2$.
\end{proof}

\begin{example}\label{ex:enriques}
There are examples of surfaces, in any nonnegative Kodaira dimension, satisfying the assumptions of Proposition \ref{prop:cat2surfaces},
as for example Enriques surfaces and classical Godeaux surfaces. For a more exhaustive list, see \cite[Ex. 6.1.14]{auel-berna-frg}.
\end{example}

\section{Categorical dimension of birational maps}

\subsection{Definition and main properties}
Let $X$ and $Y$ be smooth projective varieties.
Given a birational map $\phi: X \dashrightarrow Y$, a weak factorization $(b_1,c_1, \ldots, b_r,c_r)$ of $\phi$
is a diagram of the  form:
$$\xymatrix{
& Y_1 \ar[dl]_{b_1} \ar[dr]^{c_1} & & Y_2 \ar[dl]_{b_2} \ar[dr]^{c_2} & & & & Y_r \ar[dl]_{b_r} \ar[dr]^{c_r} &\\
X_0=X & & X_1 & & & \ldots & & & X_r=Y,
}$$
where $b_i$ and $c_i$ are either a blow-up of a smooth subvariety or an isomorphism, and $X_i$ and
$Y_i$ are smooth and projective. We also denote by $B_i \subset X_{i-1}$ the locus blown-up by $b_i$ and by
$C_i \subset X_i$ the locus blown-up by $c_i$. Recall that we assumed $k$ to have characteristic zero, so that
any birational map has a weak factorization \cite{abrmovichandco-weak}.

\begin{definition}\label{def:cat-dim-of-birmap}
Let $\phi: X \dashrightarrow Y$ be a birational map.
The motivic categorical dimension of a weak factorization $(b_1,c_1,\ldots,b_r,c_r)$ of $\phi$
is the integer 
$$\mathrm{mcd}(b_1,c_1,\ldots,b_r,c_r):= \mathrm{max}\{\mathrm{mcd}(C_i) \vert i=1,\ldots,r\}.$$

The \linedef{motivic categorical dimension} of $\phi$ is the integer
$$\mathrm{mcd} (\phi) := \mathrm{min}\{ \mathrm{mcd}(b_1,c_1,\ldots,b_r,c_r)  \, | (b_1,c_1,\ldots,b_r,c_r) \text{ is a weak factorization of } \phi \}.$$
\end{definition}

\begin{remark}\label{rmk:blow-ups}
Note that, by our convention $\mathrm{mcd}(0)=-1$, we have that $\mathrm{mcd}(\phi)=-1$ if and only if
$\phi^{-1}$ is a finite sequence of smooth blow-ups and isomorphisms.
\end{remark}

\begin{example}
Let $\sigma: X \to Y$ be the blow-up of a smooth subscheme $Z$ of $Y$ such that $\mathrm{mcd}(Z)=d$.
Then $\mathrm{mcd}(\sigma) \leq d$ and $\mathrm{mcd}(\sigma^{-1})=-1$. The second formula is obvious. To prove the second, 
consider the diagram:
$$\xymatrix{
& X \ar[dr]^\sigma \ar[dl]_{\id} \\
X \ar[rr]^\sigma & & Y,
}$$
which is a weak factorization of $\sigma$.
\end{example}

Using Bittner's presentation of the Grothendieck group $K_0(\mathrm{Var}(k))$ and Bondal-Larsen-Lunts
motivic measure, we can prove our main result.

\begin{definition}\label{def:subgrops-of-bir}
Let $X$ be a smooth projective $k$-variety, and $\mathrm{Bir}(X)$ the group of birational transformations
of $X$. We define, for any $d \geq -1$, the following subset:
$$\mathrm{Bir}_d(X) : = \{ \phi \,\, \vert \,\, \mathrm{mcd}(\phi) \leq d \} \subset \mathrm{Bir}(X).$$
\end{definition}

\begin{theorem}\label{thm:main}
The set $\mathrm{Bir}_d(X)$ is a subgroup of $\mathrm{Bir}(X)$, which coincides with the whole
group if $d \geq \dim(X)-2$, and with $\Aut(X)$ if $d=-1$.

More generally, if $\phi: X \dashrightarrow Y$ is a birational map and $I(X)=I(Y)$ in $PT(k)$, then $\mathrm{mcd}(\phi^{-1})=\mathrm{mcd}(\phi)$.
\end{theorem}

\begin{proof}
First of all, if $\phi, \psi$ are both in $\mathrm{Bir}_d(X)$, then their
compositions $\phi \circ \psi$ and $\psi \circ \phi$ also are in $\mathrm{Bir}_d(X)$. Indeed,
a weak factorization of the composition is easily written from the weak factorizations
of $\phi$ and $\psi$.

We prove now that if $\phi$ is in $\mathrm{Bir}_d(X)$, then also $\phi^{-1}$
is. We can more generally show the second statement.
So let $X$ and $Y$ be such that $I(X)=I(Y)$ and $\phi: X \dashrightarrow Y$ have motivic categorical dimension $d$. 
By definition, there is a weak factorization $(b_1,c_1,\ldots,
b_r,c_r)$ of $\phi$ with centers $C_i$ such that $I(C_i)$ is in $PT_d(k)$ for all $i$.
On the other hand, $\phi^{-1}$ clearly admits a weak factorization $(c_r,b_r,\ldots,c_1,b_1)$. 
Let us denote by $\beta_i$ (resp. $\gamma_i$)
the codimensions of $B_i$ (resp. $C_i$) in their ambient variety. Using Bittner's presentation
of $K_0(\mathrm{Var}(k))$, we obtain that 
\begin{equation}\label{eq:equality-in-K0}
[X] + \LL \sum_{i=1}^r  [B_i][\PP^{\beta_i-2}] = [Y] + \LL \sum_{i=1}^r [C_i][\PP^{\gamma_i-2}],
\end{equation}
in the Grothendieck group $K_0(\mathrm{Var}(k))$ (see, e.g., \cite{lamy-sebag}).  
Now we apply the motivic measure $\mu$ to the formula \eqref{eq:equality-in-K0}. Using the
fact that $\mu(\LL)=1$ and $\mu(\PP^m)=m+1$, we obtain the following
formula in the ring $PT(k)$:
\begin{equation}\label{eq:equality-in-PT}
I(X) + \sum_{i=1}^r (\beta_i-1)I(B_i) = I(Y) + \sum_{i=1}^r  (\gamma_i-1)I(C_i).
\end{equation}
We can cancel out $I(X)$ and $I(Y)$ in \eqref{eq:equality-in-PT} by our assumption.
This gives 
\begin{equation}\label{eq:withoutI}
\sum_{i=1}^r (\beta_i-1)I(B_i) = \sum_{i=1}^r (\gamma_i-1)I(C_i).
\end{equation}
The right hand side of \eqref{eq:withoutI} belongs to $PT_d(k)$ by assumption, so does the left hand side.
Notice that all the coefficients in both sides of \eqref{eq:withoutI} are strictly positive. It follows by 
definition of $PT_d(k)$, all of the $I(B_i)$ also belong to $PT_d(k)$, so that $\phi^{-1}$ also has
motivic categorical dimension at most $d$, and the equality follows by symmetry.

Finally, any $\phi$ in $\mathrm{Bir}(X)$ has motivic categorical dimension
bounded above by $\dim(X)-2$ since $\dim (C_i) \leq \dim(X)-2$ for any $i=1,\ldots,r$ (see Lemma \ref{lem:trivial-bound}).
On the other hand, if $\phi$ belongs to $\Bir_{-1}(X)$, so does $\phi^{-1}$. It follows from Remark \ref{rmk:blow-ups} that
this is the case if and only if $\phi$ is an automorphism.
\end{proof}

\begin{remark}
Notice that the proof of Theorem \ref{thm:main} works, more generally, whenever one has a semigroup $M \subset PT(k)$
of effective elements such that, for every $a$ in $M$, all the effective summands of $a$ belong to $M$.
In such a case, one can define $\mathrm{Bir}_M(X)$ to be the set of birational maps
admitting a weak factorization $(b_1,c_1,\ldots,b_r,c_r)$ such that all the $I(C_i)$ belong
to $M$. Then one has that $\mathrm{Bir}_M(X)$ is a subgroup of $\mathrm{Bir}(X)$. However, we stick here to
the definition of motivic categorical dimension, since semiorthogonal decompositions give all the information on
any additive invariant, and, hence, on most of the interesting
information about $X$ that one can retrieve from $I(X)$.
\end{remark}

\subsection{Cremona transformations of given motivic categorical dimension}

Here we consider $k \subset \CC$ to be an algebraically closed field.
The results treated in the above sections allow us to explain how to construct
birational maps with given motivic categorical dimension. This enables us to prove Theorem \ref{thm:introfiltraisgood}.

\begin{theorem}\label{prop:arealfiltration}
We have $\Bir_d(\PP^n) \neq \Bir_{d+1}(\PP^n)$ for all $-1 \leq d < n-2$. 
\end{theorem}
\begin{proof}
First of all, we prove that we $\Bir_{n-3}(\PP^n) \neq \Bir_{n-2}(\PP^n)$. To this
end, it is enough to construct a birational map of $\PP^n$ of motivic categorical dimension
exactly $n-2$. Let us consider homogeneous coordinates $[x_0: \ldots: x_n]$, and
a homogeneous polynomial $f$ of degree $n$ in the variables $(x_1, \ldots, x_n)$
defining a Calabi--Yau smooth hypersurface $X$  in the projective hyperplane
$H:=\{x_0=0\}$. As remarked in Example \ref{ex:CY_mfds}, we have $\mathrm{mcd}(X)=\dim(X)=n-2$,
since $h^{n-2,0}(X) \neq 0$.
Fix any homogeneous polynomial $g$ of degree $n-2$ in the variables $(x_1,\ldots,x_n)$
such that $\mathrm{gcd}(f,g)=1$, and define a birational involution $\phi_{n-2}$ of $\PP^n$ on the open subset
$x_0 \neq 0$ by the affine formula
$(x_1, \ldots, x_n) \mapsto (x_1 \frac{g}{f},\ldots, x_n \frac{g}{f})$. The fact that this gives a (birational) involution
of $\PP^n$ can be checked on the affine space $x_0 \neq 0$: iterating the formula twice, and using the fact that
$\deg(g) - \deg(f)=-2$, we have that $(\phi_{n-2})^2$ writes:
$$(x_1, \ldots, x_n) \mapsto (x_1 \frac{g}{f} \cdot \frac{g}{f} \cdot \left(\frac{g}{f}\right)^{-2} ,\ldots, x_n \frac{g}{f} \cdot \frac{g}{f} \cdot \left(\frac{g}{f}\right)^{-2})
= (x_1,\ldots,x_n),$$
that is, $(\phi_{n-2})^2$ equals the identity on the open affine subset $x_0 \neq 0$ of $\PP^n$.

In order to show that $X=\{f=x_0=0\}$ lies in the base-locus of $\phi_{n-2}$, we just have to write $\phi_{n-2}$ in homogeneous
coordinates:

$$\phi_{n-2}[x_0:\ldots:x_n] = [x_0 x_1 g : \ldots : x_0 x_n g : f].$$

Let now $(b_1,\ldots,c_r)$ be a weak factorization of $\phi_{n-2}$. In particular (see \cite[Thm. 0.1.1]{abrmovichandco-weak})
there is a $1 \leq j \leq r$ such that both $\sigma:=b_j \circ c_{j-1} \circ \ldots \circ b_1: Y_j \to \PP^n$ and
$\tau:=c_j \circ c_{j+1} \circ \ldots \circ c_r: Y_j \to \PP^n$ are projective morphisms. Since $X$ is irreducible
and has dimension $n-2$, by Zariski's Main Theorem (see, e.g., \cite[V, Thm 5.2]{hartshorne:algebraic_geometry}), 
there exists a divisor $E$ such that $\sigma(E)=X$, so that $E$ is birational to $X \times \PP^1$.
It follows that $E$ is the center $B_j$ of the blow-up $b_j$ for some $1 \leq j \leq r$, and such
$B_j$ is stably-birational to $X$. Since the numbers $h^{p,0}$ are stable birational invariants
we have $h^{n-2,0}(B_j) \neq 0$, and
then that $\mathrm{mcd}(B_j) = n-2$ by Corollary \ref{cor:maximal-motivic-dimension} and the fact that
$\dim(B_j) \leq n-2$. It follows that $\mathrm{mcd}(\phi_{n-2}^{-1}) = n-2$,
and the same holds for $\phi_{n-2}$ by Theorem \ref{thm:main}.

Now we proceed to construct a birational map $\psi$ of $\PP^n$ of motivic categorical dimension exactly $d+1$
for a given $0 \leq d \leq n-3$. 
Consider the standard birational map $\sigma: \PP^n \dashrightarrow \PP^{d+3} \times \PP^{n-d-3}$,
and let $\psi:= \sigma^{-1}\circ (\phi_{d+1},\id) \circ \sigma$, where $\phi_{d+1}$ is constructed as the above
$\phi_{n-2}$.

First of all, since $\sigma$ is obtained by resolving a linear projection, it clearly admits a weak factorization
whose centers are all (proper transforms of) linear subspaces, which have then motivic categorical dimension 0.
On the other hand, we have seen that any weak factorization of $\phi_{d+1}$ has motivic categorical dimension 
$d+1$. It follows that we can construct a weak factorization of $\psi$ of motivic categorical dimension 
$d+1$, so that $\mathrm{mcd}(\psi) \leq d+1$.

Let $X$ be the zero locus of $f$ in the corresponding linear subspace. In particular, $X$ is a $(d+1)$-dimensional
Calabi-Yau hypersurface in a hyperplane $\PP^{d+2} \subset \PP^{d+3}$, and we have $h^{d+1,0}(X) \neq 0$.
By the same argument above, since $\sigma^{-1}(X \times \PP^{d-n-3})$ is in the base locus of $\psi$,
for any weak factorization there has to be a $j$ such that $b_j$ has exceptional divisor $E$ birational
to $X \times \PP^{n-d-2}$. In particular, there is some $l$ such that $X \times \PP^{n-d-2}$ is birational to
$B_j \times \PP^l$, so that $X$ and $B_j$ are stably birational to each other.
Since the numbers $h^{p,0}$ are stable birational invariants, we have $h^{d+1,0}(B_j) \neq 0$, and
then that $\mathrm{mcd}(B_j) \geq d+1$ by Corollary \ref{cor:maximal-motivic-dimension}.
As a consequence, we have that $\mathrm{mcd}(\psi^{-1}) \geq d+1$, so that the equality holds also for $\psi$ as required.
\end{proof}

The proof of Theorem \ref{prop:arealfiltration} makes an extensive use of Hochschild homology. One
could define the \linedef{Hochschild dimension} of a pretriangulated category $\ka$ to be the maximal integer
$i$ such that $HH_i(\ka)\neq 0$, and the Hochschild dimension of a birational map by replacing the motivic categorical dimension
with the Hochschild dimension in Definition \ref{def:cat-dim-of-birmap}. The above proof would prove that this
notion also gives a proper filtration of the Cremona group.
However, the noncommutative motive is the universal additive invariant (see \cite[\S 2]{tabuada-book}), so that
motivic categorical dimension seems to be a more natural notion from a motivic point of view, and it
is also related to rationality problems (see \cite{auel-berna-frg} or \S \ref{subs:ncmrationaldefect}).
Finally, as we noticed in Proposition \ref{prop:cat2surfaces}, the notion of Hochschild dimension would be
weaker than the notion of motivic categorical dimension.

\begin{remark}\label{ex:hochdimisnotcatdim}
Suppose that $\phi: \PP^n \dashrightarrow \PP^n$ is a birational map whose base locus contains a closed subset (birational to) 
$\PP^l \times S$ with $S$ an Enriques surface. Then $\mathrm{mcd} (\phi) \geq 2$: this can be seen following the same argument in the
proof of Theorem \ref{prop:arealfiltration}, and the fact that, for surfaces, stable birationality coincides with birationality.
On the other hand, assume moreover that there exists a weak factorization $(b_1,c_1,\ldots,b_r,c_r)$
such that $HH_l(C_i)=0$ for $l\neq 0$. This is, a priori, possible, since $HH_l(S)=0$.
The birational map $\phi$ has then Hochschild dimension $0$ and motivic categorical dimension at least $2$.

It is not known to the author whether such a map $\phi$ exists, it would be
interesting to construct it, and to study the filtration on $\mathrm{Bir}_d(\PP^n)$ given by the Hochschild dimension.
\end{remark}

\section{The genus of a birational map}

%---------------------------------------------------------------------------------------------------
\subsection{Jacobians of noncommutative Chow motives}\label{sec:Jacobians}
%---------------------------------------------------------------------------------------------------
Let $k \subset \CC$ be algebraically closed.
Recall from Andr{\'e} \cite[\S4]{Andre} the construction of the category $\Chow(k)_\QQ$ of Chow motives
with rational coefficients and of the monoidal functor
$$ M(-)_\QQ:\SmProj(k)^\op \too \Chow(k)_\QQ\,,$$
where $\SmProj(k)$ is the category of smooth projective $k$-schemes.
As proved in \cite[Proposition 4.2.5.1]{Andre}, de Rham cohomology factors through Chow motives, so that 
every morphism of Chow motives induces a morphism in de Rham cohomology. For $X$ an irreducible $k$-scheme
of dimension $d$, we can consider the 
$\QQ$-vector spaces $NH_{dR}^{2i+1}(X), 0\leq i \leq d-1,$ defined by the formula
\begin{equation}\label{eq:pairings1}
\sum_C\sum_{\gamma \in \Hom(M(C)_\QQ, M(X)_\QQ(i))} \mathrm{Im}\left(H^1_{dR}(C) \stackrel{H^1_{dR}(\gamma)}{\too} H_{dR}^{2i+1}(X)\right)\,,
\end{equation}
where $C$ is a smooth and projective curve and $\gamma$ a morphism from $M(C)_\QQ$ to $M(X)_\QQ(i)$. Roughly speaking, the $NH_{dR}^{2i+1}(X)$ are the
odd pieces of de Rham cohomology that are generated by curves. Restricting the classical intersection bilinear pairings
on de Rham cohomology (see \cite[\S3.3]{Andre}) to these pieces gives pairings
\begin{eqnarray}\label{eq:pairings2}
\langle-,-\rangle: NH_{dR}^{2d-2i-1}(X) \times NH_{dR}^{2i+1}(X) \too k && 0 \leq i \leq d-1\,.
\end{eqnarray}

Recall that Marcolli and Tabuada \cite{marcolli-tabuada-jac} have defined the \linedef{Jacobian} ${\bf J}(N)$ of a noncommutative motive $N$
as an Abelian variety well-defined up to isogeny. We refrain to give the detailed definition here, the interested reader can consult \cite{marcolli-tabuada-jac}
or \cite[Ch. 7]{tabuada-book}. In particular, given a smooth and projective variety $X$ and an admissible subcategory $\ka$ of $\Db(X)$,
one can define the Jacobian ${\bf J}(\ka)$ as the Jacobian of the noncommutative motive of $\ka$ as an Abelian variety well-defined up to isogeny.

% We consider now the categories $\NChow(k)_\QQ$ and $\NNum(k)_\QQ$ of noncommutative Chow and, respectively, numerical motives with rational coefficients, see
% \cite[\S 4.1, 4.6, 4.7]{tabuada-book} for an account.
% As proved by Marcolli and Tabuada \cite{AJM},
% the category $\NNum(k)_\QQ$ is semisimple.
% 
% In order to define the Jacobian functor as done by by Marcolli and Tabuada
% \cite{marcolli-tabuada-jac}, let us recall that the category $\mathrm{Ab}(k)_\QQ$ can be identified with an abelian
% semi-simple full subcategory of $\NNum(k)_\QQ$, via natural functors 
% \begin{equation}\label{eq:ablian-to-numerical}
% \mathrm{Ab}(k)_\QQ \longrightarrow \Num(k)_\QQ \longrightarrow \NNum(k)_\QQ
% \end{equation}
% through the category of (commutative) numerical motives \cite[\S 4]{marcolli-tabuada-jac}.
% Given a noncommutative numerical motive $N$ and its decomposition in simple components
% $N \simeq S_1 \oplus \ldots \oplus S_n$

On the other hand, recall the construction of the intermediate Jacobians
$$J^i(X):= \frac{F^{i+1}H^{2i+1}(X,\CC)}{H^{2i+1}(X,\ZZ)},$$
for $0 \leq i \leq d-1$, where $F^\bullet$ is the Hodge filtration on Betti cohomology. The intermediate Jacobians are complex tori, not (necessarily) algebraic.
Nevertheless, they contain an algebraic torus $J^i_a(X) \subseteq J^i(X)$ defined as the subtorus generated by the image of
the Abel-Jacobi map
\begin{eqnarray}
AJ^i: A^{i+1}(X)_\ZZ \twoheadrightarrow J^i(X) && 0 \leq i \leq d-1\,,
\end{eqnarray}
where $A^{i+1}(X)_\ZZ$ denotes the group of algebraically trivial cycles of codimension $i+1$;
see Voisin \cite[\S 12]{voisin-book} for further details.
The algebraic intermediate Jacobian $J^i_a(X)$ is an Abelian variety, well-defined up to isogeny.

As proved in \cite[Theorem~1.7]{marcolli-tabuada-jac}, whenever the above pairings \eqref{eq:pairings2} are non-degenerate for all $i$,
one has an isomorphism 
\begin{equation}\label{eq:nc-jac=jac}
{\bf J}(\Db(X)) \simeq \prod_{i=0}^{d-1} J^i_a(X) \text{ in } \mathrm{Ab}(k)_\QQ,
\end{equation}
where $\mathrm{Ab}(k)_\QQ$ is the category of Abelian varieties up to isogeny.

As explained in
{\em loc. cit.}, \eqref{eq:pairings2} is always non-degenerate for $i=0$ and $i=d-1$. Moreover, if Grothendieck's standard
conjecture of Lefschetz type is true for $X$, then \eqref{eq:pairings2} is non-degenerate for all $i$; see Vial \cite[Lemma 2.1]{Vial}.

Since semiorthogonal decompositions induce decompositions of noncommutative motives in direct summands
(see \cite{tabuada-book}), it follows that if $\ka=\sod{\ka_1,\ldots,\ka_r}$ is a semiorthogonal decomposition, then
${\bf J}(\ka)={\bf J}(\ka_1) \oplus \ldots \oplus {\bf J}(\ka_r)$. On the other hand, since the category generated by
an exceptional object is equivalent to the category of a point, it has trivial noncommutative Jacobian.
It follows that, if $\ka$ is generated by exceptional objects, then ${\bf J}(\ka)=0$.

% We can state the following corollary.
% 
% \begin{corollary}\label{cor:mot-dim-curves}
% Let $\ka$ be a pretriangulated dg category with ${\bf J}(\ka) \neq 0$. Then ${\mathrm{mcd}}(\ka) > 0$.
% In particular, if $C$ is a curve of positive genus, then ${\mathrm{mcd}}(C)=\Repcat(C)=1$.
% \end{corollary}
% 
% Notice that Corollary \ref{cor:mot-dim-curves} can be seen as
% a particular case of Proposition \ref{prop:from-rep-to-hoch}, as one can see by the construction
% of the Jacobians and the relation \eqref{eq:hochtohodge} between Hochschild homology and Hodge
% cohomology. We want to stress it nevertheless because of its high geometrical significance.

\subsection{Weak factorizations and Jacobians}
Let $k\subset \CC$ be algebraically closed.
Recall that in our notations, for a weak factorization $(b_1,c_1,\ldots,b_r,c_r)$ we have that, for $i=1,\ldots,r$,
the map $b_i$ blows-up (possibly empty) smooth centers $B_i$ which have codimension $\beta_i$; and $c_i$ blows-up 
(possibly empty) smooth centers $C_i$ which have codimension $\gamma_i$.

\begin{proposition}\label{prop:dimension-of-jacs}
Suppose $\phi: X \dashrightarrow Y$ has a weak factorization $(b_1,c_1,\ldots,b_r,c_r)$ as above, and set $n:=\dim(X)$.
Then we have
\begin{equation}\label{eq:dims-jacs}
\dim ({\bf J}(X)) + \sum_{i=1}^r (\beta_i-1) \dim ({\bf J}(B_i))=
\dim ({\bf J}(Y)) + \sum_{i=1}^r (\gamma_i-1) \dim ({\bf J}(C_i))
\end{equation}
\end{proposition}

\begin{proof}

Let us first consider the diagram:
$$\xymatrix{
& Y_i \ar[dl]_{b} \ar[dr]^c \\
X_{i-1} & & X_{i}.
}$$
Applying blow-up formula we obtain two semiorthogonal decompositions of
$\Db(Y_i)$ inducing two decompositions of its noncommutative motive, and hence of its noncommutative Jacobian

$${\bf J}(Y_i) = {\bf J}(X_{i-1}) \oplus {\bf J}(B_i)^{\oplus \beta_i-1} = {\bf J}(X_i) \oplus {\bf J}(C_i)^{\oplus \gamma_i-1},$$

as Abelian varieties up to isogeny, as explained in \S \ref{sec:Jacobians}. Calculating dimensions,
this yields:
$$\dim ({\bf J}(X_i)) - \dim({\bf J}(X_{i-1}))=(\beta_i-1)\dim({\bf J}(B_i))-(\gamma_i-1)\dim({\bf J}(C_i)),$$
and the proof follows by summing up the above equation for $i=1,\ldots,r$.
\end{proof}

\begin{corollary}\label{cor:cor-n-1}
Let $\phi: X \dashrightarrow Y$ be a birational map as in Proposition \ref{prop:dimension-of-jacs}.
If we assume furthermore that $\dim({\bf J}(X)) = \dim({\bf J}(Y))$, then we have

$$\sum_{i=1}^r (\beta_i-1)\dim ({\bf J}(B_i))=\sum_{i=1}^r (\gamma_i-1))\dim ({\bf J}(C_i)).$$

And, assuming moreover that the pairings \eqref{eq:pairings2} are nondegenerate for $B_i$ and $C_i$ for all $i$, we obtain:

$$\sum_{i=1}^r (\beta_i-1)\dim(\prod_{l=0}^{n-\beta_i-1} J^l_a(B_i))=\sum_{i=1}^r (\gamma_i-1)\dim(\prod_{l=0}^{n-\gamma_i-1} J^l_a(C_i)).$$
\end{corollary}

\begin{proof}
The first statement is a straightforward corollary of Proposition \ref{prop:dimension-of-jacs}. For the second, just note that
the we can apply \eqref{eq:nc-jac=jac} to all of the $B_i$ and $C_i$, and formula follows from
$\dim(B_i)=n-\beta_i$ and $\dim(C_i)=n-\gamma_i$.
\end{proof}

\begin{corollary}\label{cor:sum-of-genuses}
Let $\phi: X \dashrightarrow Y$ be as in Proposition \ref{prop:dimension-of-jacs},
and assume that $\mathrm{dim}({\bf J}(X))=\mathrm{dim}({\bf J}(Y))$ and that $B_i$ and $C_i$ have dimension at most one.
This is the case in particular if $n=3$. Then the above formula \eqref{eq:dims-jacs} yields:
$$\sum_{i=1}^r  g(B_i)=
\sum_{i=1}^r g(C_i),$$
where we set, by convention, $g(Z)=0$ if $\dim(Z)=0$ or $Z$ is empty.
\end{corollary}
\begin{proof}
Notice that, under our assumptions, $B_i$ and $C_i$ have only one nontrivial Jacobian if and only if they are curves of positive genus, in which
case the dimensions of the Jacobians are $g(B_i)$
and $g(C_i)$ respectively. Moreover, in \eqref{eq:dims-jacs}, either $g(B_i)=0$ (resp. $g(C_i)=0$) or
$\beta_i=n-1$ (resp. $\gamma_i=n-1$), so that the factor $(n-2)$ can be simplified. 
\end{proof}

% \begin{remark}
% Notice that Corollary \ref{cor:sum-of-genuses} apply more generally, without simplification of the codimensional coefficients,
% to the cases where the algebraic Jacobians of the $B_i$
% and of the $C_i$ are split by Jacobians of curves, replacing $g(B_i)$ and $g(C_i)$ by the genera of the curves
% splitting the respective Jacobians.
% \end{remark}

\subsection{A generalization for the genus of birational transformations of rationally connected threefolds}

The notion of an incidence principal polarization on intermediate Jacobians was introduced by Beauville, see
\cite[\S 3.4]{beauvilleprym}. In the case where the pairings \eqref{eq:pairings2} are nondegenerate for $X$, and $X$ admits
a unique nontrivial intermediate Jacobian $J(X)$ which is principally polarized by an
incidence polarization, we say that $X$ is verepresentable. In particular, $X$ must be of odd dimension 
$2n+1$ and $J(X)=J^n(X)$.
All smooth projective curves, many Fano threefolds, all projective spaces,
quadric hypersurfaces and intersections of two quadrics, or of three even-dimensional quadrics are verepresentable. 
We refrain to give any more detail here, we refer to \cite{berna-tabuada-jacobians} for that, and we just recall that,
in these cases, the Jacobian ${\bf J}(X)$ contains
the information about the principal polarization of $J(X)$, as shown by the following theorem \cite[Thm. 1.7]{berna-tabuada-jacobians}.

\begin{theorem}\label{thm:polarization-BT}
 Let $X$ and $Y$ be verepresentable varieties, such that the pairings \eqref{eq:pairings2} are nondegenerate for both $X$
 and $Y$, and suppose that
 $$\begin{array}{cc}
\Db(X) = \sod{\ka,\kb} \,\,\,\,\, & \Db(Y) = \sod{\ka',\kc}.
   \end{array}$$
Assuming that ${\bf J}(\kb)=0$, and that $\ka \simeq \ka'$ as pretriangulated dg categories, there is an injective morphism
of principally polarized Abelian varieties $\tau: J(X) \to J(Y)$, that is $J(Y)=J(X) \oplus A$ for some Abelian variety
$A$. Moreover, if ${\bf J}(\kc)=0$ as well, $\tau$ is an isomorphism.
\end{theorem}

\begin{definition}
Let $X$ and $Y$ be verepresentable $2n+1$-folds. 
A birational map $\phi: X \dashrightarrow Y$ is \linedef{well-polarized} if there is a weak factorization of
type $(b_1,c_1,\ldots,b_r,c_r)$ such that both the centers $B_i$ and $C_i$ and all the varieties $X_i$ and $Y_i$ 
appearing in the weak factorization are verepresentable.

If $\phi$ is well-polarized by a weak factorization $(b_1,c_1,\ldots,b_r,c_r)$, we call the \linedef{Abelian type}
of $\phi$ the collection $J(C_i)$ of all nontrivial
algebraic Jacobians of the centers $C_i$. If $S_A$ is a set of indecomposable principally polarized Abelian varieties
such that all of the $J(C_i)$ are isomorphic as principally polarized Abelian varieties to 
sums of elements of $S_A$, we say that the Abelian type of the weak factorization is \linedef{split by $S_A$}.
If, moreover, $S_A$ only contains Jacobian of curves, we say that the Abelian type is \linedef{Jacobian}.
\end{definition}

\begin{remark}
Notice that the weak factorization $(b_1,c_1,\ldots,b_r,c_r)$ of $\phi$ is well-polarized, then the centers $B_i$ and $C_i$
either have trivial Jacobian or have codimension $2$. Indeed, blowing-up $B_i$ adds $\beta_i-1$ copies of the Hodge structure
of $B_i$ to the Hodge structure of $X_{i-1}$. So that if $\beta_i >2$, we need the Jacobian of $B_i$ to be trivial if we want $Y_i$
to be verepresentable. A similar argument works for $C_i$.
\end{remark}

\begin{remark}\label{rem:3fold-polarized}
In the case $n=3$, and $X$ and $Y$ verepresentable varieties, every weak factorization of a birational map $\phi: X \dashrightarrow Y$ is
well-polarized of Jacobian Abelian type. Indeed, the centers are either points or smooth projective curves.
\end{remark}

The next Proposition uses Theorem \ref{thm:polarization-BT} to produce subgroups of $\mathrm{Bir}(X)$. Notice that
similar arguments were used by Clemens and Griffiths \cite{clemens_griffiths} to show that any rational complex threefold has 
its intermediate
Jacobian split by Jacobians of curves as a principally polarized Abelian variety. A categorical rephrasing of this criterion can be found in
\cite{bolo_berna:conic} (see also \cite{bolognesi_bernardara:representability} and \cite{auel-berna-frg}).

\begin{proposition}\label{prop:filtration-by-SA}
Let $X$ be a verepresentable $(2n+1)$-fold. Given a set of indecomposable principally polarized Abelian varieties $S_A$,
the set of birational transformations $\phi$ well-polarized by weak factorizations of Abelian type split by $S_A$ is a 
subgroup of $\mathrm{Bir}(X)$.
\end{proposition}

\begin{proof}
By assumption  there is a weak factorization $(b_1,c_1,\ldots,b_r,c_r)$ 
of $\phi$ such that
all the varieties involved are verepresentable and the intermediate Jacobians of the $C_i$ are direct sums of 
elements of $S_A$.
Consider the first diagram:
$$\xymatrix{X & Y_1 \ar[l]_{b_1} \ar[r]^{c_1} & X_1.}$$
By assumption, we have
\begin{equation}\label{eq:first-jacobians}
J(X) \oplus J(B_1) = J(X_1) \oplus J(C_1)
\end{equation}
as principally polarized Abelian varieties. If $r=1$, that is, if $X_1=Y$, we have $J(X_1)=J(X)$ by assumption and we can conclude.
Otherwise, consider the next diagram
$$\xymatrix{X_1 & Y_2 \ar[l]_{b_2} \ar[r]^{c_2} & X_2,}$$
from which we obtain
\begin{equation}\label{eq:second-jacobians}
J(X_1) \oplus J(B_2) = J(X_2) \oplus J(C_2).
\end{equation}
Combining \eqref{eq:first-jacobians} and \eqref{eq:second-jacobians} we obtain:
$$J(X) \oplus J(B_1) \oplus J(B_2) = J(X_1) \oplus J(C_1) \oplus J(B_2) = J(X_2) \oplus J(C_1) \oplus J(C_2)$$
as principally polarized Abelian varieties. Recursively, running over the whole weak factorization, we obtain
$$J(X) \oplus \bigoplus_{i=1}^r J(B_i) = J(Y) \oplus \bigoplus_{i=1}^r J(C_i)$$
and hence, since $J(X)=J(Y)$ and the category of principally polarized Abelian varieties is semisimple,
the Jacobians $J(B_i)$ are also split by elements of $S_A$.

Finally, let $\phi$ and $\psi$ be well-polarized by weak factorizations of Abelian type split by $S_A$.
Then also $\phi \circ \psi$ is well polarized since it admits a weak factorization
which is just given by juxtaposition of the weak factorizations of $\phi$ and $\psi$, and it is clear that 
such a factorization has Abelian 
type split by $S_A$.
\end{proof}

\begin{question}
Construct nontrivial examples of well-polarized birational maps which are not isomorphisms, admitting a weak factorization whose Abelian type is
not split by the set $\{0\}$ for varieties of dimension $\geq 5$, for example
for $\PP^5$. Note that the quadro-quadric Cremona transformations of $\PP^n$ described in \cite{ein-sheperdB} for $n=5,8,14$ are not good candidates
since they admit a factorization involving only varieties with $hh_i=0$ for $i \neq 0$ (and most likely also the one for $n=26$). An example
must involve varieties with nontrivial Jacobian, such as te cubic threefold, for example. The difficult step here would be to show that \it any \rm
weak factorization must involve varieties with such Jacobian.
\end{question}

\begin{definition}
If $\phi$ is well-polarized by a weak factorization $(b_1,c_1,\ldots,b_r,c_r)$, 
the maximum $\mathrm{max}\{\mathrm{dim}(J(C_i))\}$ over all the centers $C_i$ is called the genus of the given
weak factorization. The \linedef{genus} $g(\phi)$ of
the birational map $\phi$ is the smallest genus of any possible well-polarized weak factorization.
\end{definition}

Let $X$ and $Y$ be verepresentable threefolds, and $\phi:X\dashrightarrow Y$ a birational map. 
Recall that by \linedef{regular resolution} of $\phi^{-1}$, we mean an iterated blow-up
$\sigma: Y' \to Y$ along smooth centers such that there is a birational morphism $\rho: Y' \to X$ and the 
composition $\rho\circ \sigma^{-1} = \phi^{-1}$ is $\phi^{-1}$.
Frumkin \cite{frumkin} defined a notion of genus of $\phi$, which we denote by $g_F(\phi)$, to be the maximum of the genera among
the centers of the blow-ups in regular resolutions of $\phi^{-1}$. As proved by Frumkin \cite[Corollary 1.2.4]{frumkin},
this does not depend on the choice of the regular resolution. On the other hand, Lamy defines the genus of a birational
map to be the maximal genus of a curve $C$ such that there exist a divisor $D$ in $X$ contracted by $\phi$, birational to $C\times \PP^1$,
and shows that his definition coincides with Frumkin's one \cite{lamy-genus}.

\begin{proposition}\label{prop:frumkin}
Let $X$ and $Y$ be verepresentable threefolds and $\phi: X \dashrightarrow Y$ a birational map. Then $\phi$ is well-polarized
and any weak factorization of $\phi$ has Jacobian Abelian type. Moreover, $g_F(\phi) \leq g(\phi)$.
\end{proposition}
\begin{proof}
First of all, it is easy to see that any weak factorization of $\phi$ is a composition of blow-ups along curves or points.
It follows that $\phi$ is well polarized and the weak factorization is of Jacobian Abelian type.

Suppose that we are given $\phi$ with genus $g_F(\phi) >0$. Then any regular resolution of $\phi$ blows up a curve $C$ with $g(C) = g_F (\phi)$. We can now proceed
as in the proof of Theorem \ref{prop:arealfiltration} to show that if $(b_1,c_1,\ldots,b_r,c_r)$ is a weak factorization of $\phi$,
then there must be an $i$, $1\leq i \leq r$ such that $g(C_i)=g(C)$. We have then that $g_F(C) \leq g(\phi)$. 
\end{proof}

\begin{question}\label{question:genera}
Let $\phi$ be as in Proposition \ref{prop:frumkin}. Do we have $g(\phi)=g_F(\phi)$? 
\end{question}

One would expect a positive answer to Question \ref{question:genera}, so that the genus defined via the weak factorization
would be indeed a generalization of Frumkin's genus to higher dimensional cases. However, a negative answer would be a very
interesting and deep result. Let us sketch a little argument reducing the question to maps with trivial Frumkin genus.

Let $\phi$ be as in Proposition \ref{prop:frumkin}, and consider a regular resolution $X \stackrel{\rho}{\leftarrow} Y' \stackrel{\sigma}{\to} Y$
of $\phi^{-1}$, and a regular resolution
$Y \stackrel{\rho'}{\leftarrow} X' \stackrel{\sigma'}{\to} X$ of $\phi$, and the commutative diagram:
$$\xymatrix{
X' \ar[d]_{\sigma'} \ar@{-->}[r]^\psi & Y' \ar[d]^\sigma \\
X \ar@{-->}[r]^\phi & Y.
}$$
In particular, thanks to \cite[Prop. 2.2]{frumkin} (see also \cite[Prop. 7]{lamy-genus}), the map $\psi$ has genus
zero, hence there is a regular resolution $Y' \stackrel{f}{\leftarrow} X'' \stackrel{\tau'}{\to} X'$ of $\psi$ 
where $\tau'$ blows-up points and rational curves. If one was able to construct
a weak factorization of $f$ whose centers are points or rational curves only, then the juxtaposition of $\sigma'$, $\tau'$, such
weak factorization, and $\sigma$ would be a weak factorization of $\phi$ with genus exactly $g_F(\phi)$. Hence, question
\ref{question:genera} reduces to the following question: suppose $\phi: X \to Y$ is a birational morphism between verepresentable
threefolds with $g_F(\phi)=0$. Is it possible to construct a weak factorization of $\phi$ of genus zero?

Notice that, using Lamy's definition of $g_F(\phi)$, one can show that the irreducible components of the
exceptional locus of such a $\phi$ are either rational curves of rational surfaces. If one was able to construct a map $\phi$
such that $\phi$ has no weak factorization of genus zero, we would have a negative answer to Question \ref{question:genera}.
This would be a quite interesting example, since it would show that the group of genus zero maps is strictly smaller
than the group of the maps whose exceptional locus has rational connected components. On the other hand, as 
we will see in the next Section, all the elements of the subgroup of $\Bir(\PP^3)$ generated by the standard Cremona transformation
have genus zero.

\section{Toric versus rational centers}\label{sec:refinements}

\subsection{The noncommutative motivic rational defect and $\mathrm{Bir}_{n-4}(X)$.}\label{subs:ncmrationaldefect}
Let $Y$ be a rational $n$-dimensional variety. As remarked in \cite{galkin-shinder-cubic},
from the existence of weak factorizations for birational maps, one has
\begin{equation}\label{eq:mot-rat-defect}
[Y] = [\PP^n] + \LL M_Y
\end{equation}
in $K_0(\mathrm{Var}(k))$, where $M_Y$ is a $\ZZ$-linear combination of classes of varieties of
dimension bounded above by $n-2$. Galkin
and Shinder define then $([Y]-[\PP^n])/\LL \in K_0(\mathrm{Var}(k))[\LL^{-1}]$ as the
\linedef{rational defect} of $Y$ \cite{galkin-shinder-cubic}.
Applying to \eqref{eq:mot-rat-defect} the motivic measure $\mu$ defined in \eqref{eq:def-of-mbll},
and recalling that $\mathrm{mcd}(\PP^n)=0$, we obtain the following statement (see
\cite[Prop. 8.1.2]{auel-berna-frg}).

\begin{proposition}\label{prop:motivic-nv-obstruct}
If $Y$ is rational smooth, projective, then $\mathrm{mcd}(Y) \leq \mathrm{max}\{0,n-2\}$. 
\end{proposition}

We call the class of $I(Y)$ in the $\ZZ$-module $PT(k)/PT_{n-2}(k)$ the \linedef{noncommutative motivic rational defect} of
the variety $Y$.

\begin{definition}
We say that a birational map $\phi: X \dashrightarrow Y$ has \linedef{rational centers} if there exists a weak factorization
$(b_1,c_1,\ldots,b_r,c_r)$ with $C_i$ rational for any $i=1,\ldots,r$.
For a given smooth and projective variety $X$, we define $R_X$ to be the subgroup of $\Bir(X)$ generated by birational 
maps with rational centers.
\end{definition}

\begin{corollary}\label{cor:rational-centers-dimn-4}
Let $\phi: X \dashrightarrow Y$ be a birational map with rational centers, and $n=\dim(X)$. Then $\mathrm{mcd}(\phi) \leq \max\{0,n-4\}$.
In particular, $R_X \subset \Bir_{n-4}(X)$ if $n\geq 4$, and $R_X=\Bir_0(X)$ if $n\leq3$.
\end{corollary}

\begin{proof}
By proposition \ref{prop:motivic-nv-obstruct} a rational variety $C$ of 
dimension $m$ has $\mathrm{mcd}(C) \leq m-2$. If follows then, that,
in the above assumptions, $\mathrm{mcd}(C_i) \leq \dim(C_i)-2 \leq n-4$
for any $i$, and the first statement is proved. Also, this proves that
$R_X$ is indeed contained in $\Bir_{n-4}(X)$ if $n \geq 4$.

The stronger statement in the case $n=3$ is just a rephrasing of the fact that
having rational centers is equivalent to having motivic dimension zero. Indeed, in this
case, the centers of a weak factorization are either points or smooth curves, and we know that
a curve $C$ has $\mathrm{mcd}(C)=0$ if and only if $C$ is rational.

The statement is trivial for $n \leq 2$.
\end{proof}

In the case of fourfolds, centers have dimension at most two.
A conjecture attributed to Orlov (see, e.g., \cite[Conj. 6.2.1]{auel-berna-frg})
claims that a smooth projective complex surface is rational if and only if it has a full
exceptional collection. We can formulate therefore the following Conjecture.

\begin{conjecture}\label{conj:orlov}
Suppose that $\dim(X) \leq 4$, and let $\phi: X \dashrightarrow X$ be a birational transformation.
Then $\mathrm{mcd}(\phi)=0$ if and only if $\phi$ has rational centers.
\end{conjecture}

We notice that Hochschild homology is certainly not fine enough to study the above Conjecture, as Proposition
\ref{prop:cat2surfaces} shows.

On the other hand, it is easy to see that there are rational threefolds $X$ such that $\mathrm{mcd}(X) = 1$, if
for example, $X$ satisfies $h^{1,2}(X)\neq 0$.
Suppose then that $\phi$ is a birational map of $\PP^n$ contracting some rational variety $Z$ with $\mathrm{mcd}(Z) > 0$.
One would be tempted to deduce that $\mathrm{mcd}(\phi)>0$, but the arguments used in the proof of Theorem \ref{prop:arealfiltration}
are deeply based on birational geometry. We formulate then the following question.

\begin{question}\label{quest:conject}
Let $n \geq 5$. Does there exist $\phi$ in $R_{\PP^n}$ such that $\mathrm{mcd}(\phi) > 0$?
\end{question}

\subsection{Toric centers, and the standard Cremona transformation}
Let $X$ be a smooth projective variety. We say that a birational transformation $\phi: X \dashrightarrow X$ has \linedef{toric centers}
if it admits a weak factorization $(b_1,c_1,\ldots,b_r,c_r)$ where all the 
$C_i$ are toric. By this, we mean that $C_i$ is abstractly isomorphic, as a smooth projective variety, to a toric
variety. We set $T_X$ to be the subgroup of $\Bir(X)$ generated by birational maps with toric centers.
For example, a birational map $\phi: X \dashrightarrow X$
admitting a weak factorization such that all the $C_i$ are projective spaces
has toric centers.
Notice however that $\phi$ having toric centers does not imply that $X$ itself is toric, nor 
that any of the blow-ups is a toric map.
By Kawamata \cite{kawamata-toric}, the derived category of any toric variety is generated by exceptional objects.
We therefore have the following easy remark.

\begin{proposition}\label{prop:toriccenters}
If $\phi: X \dashrightarrow Y$ has toric centers, then $\mathrm{mcd}(\phi)=0$.
In particular, $T_X \subset \Bir_0(X)$.
\end{proposition}

Notice that we do not know whether $\phi$ having toric centers implies that
$\phi^{-1}$ has toric centers. It is not difficult indeed to have non-toric varieties
whose derived categories are generated by exceptional objects, as for example any del Pezzo
surface of degree at most 5, or, more generally, any non-toric blow-up along points of a 
variety whose derived category is generated by exceptional objects.

A very well known example of a birational map with toric centers is given by the
standard Cremona transformation $\sigma_n: \PP^n \dashrightarrow \PP^n$ which admits a factorization:
$$\xymatrix{
& X \ar[dl]_b \ar[dr]^c \\
\PP^n \ar@{-->}[rr]^{\sigma_n} & & \PP^n
}$$
as follows: there are $n+1$ points in general position on $\PP^n$ naturally associated to $\sigma_n$. The map $c$
is the composition of the blow-ups of these points, followed by the blow-ups of the strict transforms of all the lines through two of them,
then by the blow-ups of the strict transforms of planes through three of them and so on, the last step being the blows-up of the strict transforms of
the $(n-2)$-dimensional linear subspaces spanned by all the possible choices of $n-1$ points. Notice that $X$ is also known as the
Losev--Manin space, a compactification of the moduli space of $(n+3)$-pointed curves of genus zero. It follows
that all the standard Cremona transformations have toric centers, and are indeed toric. We remark that here we could have
used that projective spaces have motivic categorical dimension $0$, without appealing to toric varieties.

We have the following Corollary of Proposition \ref{prop:toriccenters}.
We set $G_n:=\sod{\mathrm{PGL}_{n+1}(k),\sigma_n}$ as a subgroup of $\Bir(\PP^n)$.

\begin{corollary}\label{cor:cremona-is-in-bir0}
For any $n\geq 2$, we have
$$G_n \subset T_{\PP^n} \subset \mathrm{Bir}_0(\PP^n).$$
\end{corollary}

We finally notice, as a possible application of Corollary \ref{cor:cremona-is-in-bir0},
that a positive answer to Question \ref{quest:conject} would imply that the group
generated by maps contracting rational varieties is strictly bigger than $G_n$, a result which
was proved by Blanc and Hed\'en \cite{blanc-heden} if $n \geq 3$ is odd.

\end{document}